\documentclass[11pt]{amsproc}
\usepackage{amssymb,amsmath,amsthm,anysize}
\usepackage{longtable}
\usepackage{booktabs}
\usepackage{multirow,bigstrut,bigdelim,color}
\usepackage[shortlabels]{enumitem}

\marginsize{2cm}{2cm}{2cm}{2cm}

\newtheorem{Thm}{Theorem}[section]
\newtheorem{Cor}[Thm]{Corollary}

\newtheorem{La}[Thm]{Lemma}

\newtheorem{Claim}[Thm]{Claim}

\newenvironment{Prf}{\noindent\textbf{Proof.}}{\hfill $\Box$ \medskip}

\newcommand{\Inn}{\operatorname{Inn}}
\newcommand{\Aut}{\operatorname{Aut}}
\newcommand{\Out}{\operatorname{Out}}
\newcommand{\Hol}{\operatorname{Hol}}
\newcommand{\PG}{\operatorname{PG}}
\newcommand{\Alt}{\mathrm{Alt}}
\newcommand{\inc}{\,\mathrm{I}\,}

\renewcommand\le{\leqslant}
\renewcommand\ge{\geqslant}

\setlist[description]{leftmargin=1ex,font=\normalfont\bfseries\space,style=nextline}

\begin{document}

\title[Primitive generalised quadrangles of holomorph type]{Point-primitive, line-transitive generalised quadrangles of holomorph type}
\author{John Bamberg, Tomasz Popiel, Cheryl E. Praeger}

\address{
Centre for the Mathematics of Symmetry and Computation\\
School of Mathematics and Statistics\\
The University of Western Australia\\
35 Stirling Highway, Crawley, W.A. 6009, Australia.\newline
Email: \texttt{\{john.bamberg, tomasz.popiel, cheryl.praeger$^\dag$\}@uwa.edu.au}\\
\newline $^\dag$ Also affiliated with King Abdulaziz University, Jeddah, Saudi Arabia.
}

\thanks{The first author acknowledges the support of the Australian Research Council (ARC) Future Fellowship FT120100036. 
The second author acknowledges the support of the ARC Discovery Grant DP1401000416. 
The research reported in the paper forms part of the ARC Discovery Grant DP1401000416 of the third author.}



\maketitle

\begin{abstract}
Let $G$ be a group of collineations of a finite thick generalised quadrangle $\Gamma$. 
Suppose that $G$ acts primitively on the point set $\mathcal{P}$ of $\Gamma$, and transitively on the lines of $\Gamma$. 
We show that the primitive action of $G$ on $\mathcal{P}$ cannot be of holomorph simple or holomorph compound type. 
In joint work with Glasby, we have previously classified the examples $\Gamma$ for which the action of $G$ on $\mathcal{P}$ is of affine type. 
The problem of classifying generalised quadrangles with a point-primitive, line-transitive collineation group is therefore reduced to the case where there is a unique minimal normal subgroup $M$ and $M$ is non-Abelian.
\end{abstract}

\section{Introduction}

A {\em partial linear space} is a point--line incidence geometry in which any two distinct points are incident with at most one line. 
All partial linear spaces considered in this paper are assumed to be finite. 
A {\em generalised quadrangle} $\mathcal{Q}$ is a partial linear space that satisfies the {\em generalised quadrangle axiom}: given a point $P$ and line $\ell$ not incident with $P$, there is a unique line incident with $P$ and concurrent with $\ell$. 
This axiom implies, in particular, that $\mathcal{Q}$ contains no triangles. 
If each point of $\mathcal{Q}$ is incident with at least three lines, and each line is incident with at least three points, then $\mathcal{Q}$ is said to be {\em thick}. 
In this case, there exist constants $s,t \ge 2$ such that each point (line) is incident with exactly $t+1$ lines ($s+1$ points), and $(s,t)$ is called the {\em order} of $\mathcal{Q}$. 
Generalised quadrangles were introduced by Tits~\cite{Tits:1959cl}, together with the other {\em generalised polygons}, in an attempt to find a systematic geometric interpretation for the simple groups of Lie type. 
It is therefore very natural to ask which groups arise as collineation groups of generalised quadrangles.

A topic of particular interest is that of generalised quadrangles admitting collineation groups $M$ that act {\em regularly} on points, where the point set is identified with $M$ acting on itself by right multiplication. 
Ghinelli~\cite{MR1153980} showed that a Frobenius group or a group with non-trivial centre cannot act regularly on the points of a generalised quadrangle of order $(s,t)$ if $s$ is even and $s=t$, and Yoshiara~\cite{MR2287459} showed that a generalised quadrangle with $s=t^2$ does not admit a point-regular collineation group. 
Regular groups arise, in particular, as subgroups of certain {\em primitive} groups. 
Bamberg et~al.~\cite{Bamberg:2012yf} showed that a group $G$ acting primitively on both the points and the lines of a generalised quadrangle must be almost simple. 
The present authors and Glasby~\cite[Corollary~1.5]{OurAffine} sought to weaken this assumption to primitivity on points and {\em transitivity} on lines, and, using a result of De Winter and Thas~\cite{MR2201385}, classified the generalised quadrangles admitting such a group in the case where the primitive action on points is of {\em affine} type. (There are only two examples, arising from {\em hyperovals} in $\PG(2,4)$ and $\PG(2,16)$.) 
In this case, the regular subgroup $M$ of $G$ is Abelian, and hence {\em left} multiplication by any element of $M$ is also a collineation. 
We consider the situation where $M$ is non-Abelian but $G$ has a second minimal normal subgroup, which is necessarily the centraliser of $M$, so that all left multiplications are again collineations. 
In the context of the O'Nan--Scott Theorem~\cite[Section 5]{MR1477745} for primitive permutation groups, this means that the action of $G$ on points is of either holomorph simple (HS) or holomorph compound (HC) type (see Section~\ref{sec:results} for definitions). 
We prove the following result.

\begin{Thm} \label{mainGQthm}
Let $G$ be a collineation group of a finite thick generalised quadrangle with point set $\mathcal{P}$ and line set $\mathcal{L}$. 
If $G$ acts transitively on $\mathcal{L}$ and primitively on $\mathcal{P}$, then $G$ has a unique minimal normal subgroup; that is, the action of $G$ on $\mathcal{P}$ does not have O'Nan--Scott type HS or HC.
\end{Thm}

The proof of Theorem~\ref{mainGQthm} is given in Sections~\ref{sec:proof} and~\ref{sec:proofHC}, using some preliminary results established in Section~\ref{sec:results}, and the Classification of Finite Simple Groups.

\section{Preliminaries} \label{sec:results}

We first recall some definitions and facts about permutation groups. 
Let $G$ be a group acting on a set $\Omega$, and denote the image of $x\in \Omega$ under $g\in G$ by $x^g$. 
The {\em orbit} of $x\in \Omega$ under $G$ is the set $x^G = \{ x^g \mid g\in G \}$, the subgroup $G_x = \{ g\in G \mid x^g=x \}$ is the {\em stabiliser} of $x\in \Omega$, and the {\em Orbit--Stabiliser Theorem} says that $|G:G_x| = |x^G|$. 
The action of $G$ is {\em transitive} if $x^G=\Omega$ for some (and hence every) $x\in\Omega$, and {\em semiregular} if $G_x$ is trivial for all $x\in \Omega$. 
It is {\em regular} if it is both transitive and semiregular. 
If $G$ acts transitively on $\Omega$ and $M$ is normal subgroup of $G$, then all orbits of $M$ on $\Omega$ have the same length, and in particular it makes sense to speak of $M$ being semiregular.

Given $g \in G$, define $\rho_g$, $\lambda_g,\iota_g \in \text{Sym}(\Omega)$ by $\rho_g : x \mapsto xg$, $\lambda_g : x \mapsto g^{-1}x$, and $\iota_g : x \mapsto g^{-1}xg$. 
Set $G_R = \{ \rho_g : g \in G \}$, $G_L = \{ \lambda_g : g \in G \}$, and $\Inn(G) = \{ \iota_g : g \in G \}$. 
The {\em holomorph} $\Hol(G)$ of $G$ is the semidirect product $G_R \rtimes \Aut(G)$ with respect to the natural action of $\Aut(G)$ on $G_R$ \cite[Section~2.6]{JohnThesis}. 
We have $\Hol(G) = N_{\text{Sym}(G)}(G_R)$, and $G_L = C_{\text{Sym}(G)}(G_R)$. 
A group $H$ acting on a set $\Delta$ is {\em permutationally isomorphic} to $G$ acting on $\Omega$ if there is an isomorphism $\theta : G \rightarrow H$ and a bijection $\beta : \Omega \rightarrow \Delta$ such that $\beta(\omega^g) = \beta(\omega)^{\theta(g)}$ for all $g\in G$ and $\omega \in \Omega$. 
If a group $M$ acts regularly on $\Omega$, then there is a permutational isomorphism $\theta : N_{\text{Sym}(\Omega)}(M) \rightarrow \Hol(M)$ with bijection $\beta : \Omega \rightarrow M$, where $\beta : \alpha^g \mapsto g$ for some fixed $\alpha \in \Omega$, and $\theta : \tau \mapsto \beta^{-1} \tau \beta$. 
We have $\theta(M) = M_R$, so the regular action of $M$ on $\Omega$ is permutationally isomorphic to the action of $M$ on itself by right multiplication, and hence we can identify $\Omega$ with $M$. 
Furthermore, $\theta(C_{\text{Sym}(\Omega)}(M)) = M_L$. 
If $M$ is a normal subgroup of $G$, then $G$ is permutationally isomorphic to a subgroup of $\Hol(M)$. 
If $M \rtimes \Inn(M) \le G$, then $G$ contains $M_L$ because $M \rtimes \Inn(M) = \langle M_R,M_L \rangle$.

A transitive action of $G$ on $\Omega$ is said to be {\em primitive} if it preserves no non-trivial partition of $\Omega$. 
The structure of a primitive permutation group is described by the O'Nan--Scott Theorem~\cite[Section 5]{MR1477745}, which splits the primitive permutation groups into eight types. 
We are concerned with only two of these types. 
If $M \rtimes \Inn(M) \le G \le M \rtimes \Aut(M)$ with $M \cong T$ for some non-Abelian finite simple group $T$, then $G$, being contained in the holomorph of a simple group, is said to have type {\em HS}. 
If instead $M$ is isomorphic to a compound group $T^k$, $k\ge 2$, then $G$ has type {\em HC}. 
In this case, $G$ induces a subgroup of $\text{Aut}(M)\cong \text{Aut}(T) \wr S_k$ which acts transitively on the set of $k$ simple direct factors of $M\cong T^k$. 
In either case, $G$ contains $M_R$ and $M_L$, as explained above. 

If we write $\mathcal{S}=(\mathcal{P},\mathcal{L},\inc)$ for a partial linear space, then we mean that $\mathcal{P}$ is the point set, $\mathcal{L}$ is the line set, and $\inc$ is the incidence relation. 
An incident point--line pair is called a {\em flag}. 
A {\em collineation} of $\mathcal{S}$ is a permutation of $\mathcal{P}$, together with a permutation of $\mathcal{L}$, such that incidence is preserved. 
If $\mathcal{S}$ admits a group of collineations $M$ that acts regularly on $\mathcal{P}$, then we identity $\mathcal{P}$ with $M$ acting on itself by right multiplication (as above). 
A line $\ell$ is then identified with the subset of $M$ comprising all of the points incident with $\ell$, and hence $P \inc \ell$ if and only if $P \in \ell$. 
Moreover, the stabiliser $M_\ell$ is the set of all elements of $M$ that fix $\ell$ setwise by right multiplication. 

\begin{La} \label{lineUnionCosets}
Let $\mathcal{S}=(\mathcal{P},\mathcal{L},\inc)$ be a partial linear space with no triangles, and let $G$ be a group of collineations of $\mathcal{S}$ with a normal subgroup $M$ that acts regularly on $\mathcal{P}$. 
Let $\ell$ be a line incident with the identity $1 \in M = \mathcal{P}$, and suppose that its stabiliser $M_\ell$ is non-trivial. 
Then
\begin{itemize}
\item[\textnormal{(i)}] $\ell$ is a union of left $M_\ell$-cosets, including the trivial coset;
\item[\textnormal{(ii)}] if $M \rtimes \Inn(M) \le G$, then $M_\ell=\ell$.
\end{itemize}
\end{La}

\begin{Prf}
(i)
Let $g\in M_\ell$. Since $1\inc\ell$, namely $1\in \ell$, we have $g=1^g\inc \ell^g=\ell$, namely $g\in \ell$.
Therefore, $M_\ell \subseteq \ell$. 
Now, if $h \not \in M_\ell \backslash \{1\}$ is incident with $\ell$, then every non-trivial element of $M_\ell$ must map $h$ to another point incident with $\ell$, and hence the whole coset $hM_\ell$ is contained in $\ell$. 

(ii)
By (i), $M_\ell \subseteq \ell$, so it remains to show the reverse inclusion.
Let $m \in \ell \setminus \{1\}$. 
Since $M_\ell$ is non-trivial, there exists a non-trivial element $h\in M_\ell$. 
Since $M \rtimes \Inn(M) \le G$, left multiplication by $h^{-1}$ is a collineation of $\mathcal{S}$. 
Since $1$ and $m$ are both incident with $\ell$, it follows that $h^{-1}$ and $h^{-1}m $ are collinear. 
On the other hand, $h^{-1}\in M_\ell \subseteq \ell$ by (i), so $h^{-1}m$ is collinear with $m$ because right multiplication by $m$ is a collineation. 
That is, $h^{-1}m$ is collinear with two points $h^{-1},m$ that are incident with $\ell$, and so $h^{-1}m$ is itself incident with $\ell$ because $\mathcal{S}$ contains no triangles. 
Therefore, $m$ maps two points $1,h^{-1}$ incident with $\ell$ to two points $m,h^{-1}m$ incident with $\ell$, and so $m \in M_\ell$. 
\end{Prf}

\begin{Thm} \label{prop:flagTrans}
Let $\mathcal{S} = (\mathcal{P},\mathcal{L},\inc)$ be a partial linear space with no triangles. 
Let $G$ be a group of collineations of $\mathcal{S}$ that acts transitively on $\mathcal{L}$, and suppose that $G$ has a normal subgroup $M$ that acts regularly on $\mathcal{P}$ and satisfies $M \rtimes \Inn(M) \le G \le M \rtimes \Aut(M)$. 
If the action of $M$ on $\mathcal{L}$ is not semiregular, then the lines $\ell_1,\ldots,\ell_{t+1}$ incident with $1$ are a $G_1$-conjugacy class of subgroups of $M$, and $G$ acts transitively on the flags of $\mathcal{S}$.
\end{Thm}

\begin{Prf} 
Since $M$ acts transitively on $\mathcal{P}$, we have $G = M G_1 = G_1 M$. 
By assumption, $G \le \Hol(M)$ and so $G_1 \le \Aut(M)$. 
By Lemma~\ref{lineUnionCosets}(ii), the lines $\ell_1,\ldots,\ell_{t+1}$ can be identified with subgroups of $M$.
Each $g \in G_1$, acting naturally as an element of $\Aut(M)$, fixes $1$ and hence maps $\ell_1$ to $\ell_1^g = \ell_i$ for some $i \in \{1,\ldots,t+1\}$. 
Conversely, consider the map $\varphi : G \rightarrow \text{Aut}(M)$ defined by $\varphi(g) = \iota_g$. 
The restriction of $\varphi$ to $G_1$ is the identity. 
Moreover, $\operatorname{ker}(\varphi) = C_G(M)$, and hence $\theta(\operatorname{ker}(\varphi)) = M_L$, where $\theta$ is the permutational isomorphism defined above. 
In particular, $\operatorname{ker}(\varphi)$ acts transitively (indeed, regularly) on $\mathcal{P}$.
Hence, $\operatorname{ker}(\varphi) G_1 = G$, so $\operatorname{Im}(\varphi) = \varphi(G_1) = G_1$. 
Now consider a line $\ell_i$ for some $i>1$. 
By line-transitivity, $\ell_i=\ell_1^g$ for some $g\in G$. 
On the other hand, since $G = \operatorname{ker}(\varphi) G_1$, we have $g=zg_1$ for some $z\in\operatorname{ker}(\varphi)$ and $g_1\in G_1$, so $\ell_1^g = \ell_1^{g_1}$. 
Therefore, $\ell_1,\ldots,\ell_{t+1}$ are precisely the subgroups of the form $\ell_1^g$ with $g\in G_1$. 
Since the lines $\ell_i$ and $\ell_j$ intersect precisely in the point $1$ for $i\neq j$, the $t+1$ subgroups $\ell_1,\ldots,\ell_{t+1}$ are distinct, and they form a single $G_1$-conjugacy class of subgroups of $M$. 
In particular, $G_1$ acts transitively on $\{\ell_1,\ldots,\ell_{t+1}\}$, so $G$ acts transitively on the flags of $\mathcal{S}$.
\end{Prf}

Let us draw a corollary in the case where $\mathcal{S}$ is a thick generalised quadrangle. 
In this case, $\mathcal{S}$ has $(s+1)(st+1)$ points and $(t+1)(st+1)$ lines, where $(s,t)$ is the order of $\mathcal{S}$.

\begin{Cor} \label{GQcor}
If the partial linear space in Theorem~$\ref{prop:flagTrans}$ is a thick generalised quadrangle of order $(s,t)$, then $s+1$ divides $t-1$.
\end{Cor}

\begin{Prf}
Begin by observing that $\Inn(M)$ acts on $\{ \ell_1,\ldots,\ell_{t+1} \}$. 
That is, for each $g \in M$, we have $g^{-1} \ell_1 g = \ell_i$ for some $i \in \{1,\ldots,t+1\}$. 
Suppose first that $\Inn(M)$ is intransitive on $\{ \ell_1,\ldots,\ell_{t+1} \}$. 
Then, without loss of generality, $\ell_2$ is in a different $\Inn(M)$-orbit to $\ell_1$, and so, for every $g \in M$, we have $g^{-1} \ell_1 g = \ell_i$ for some $i \neq 2$.  
Hence, every double coset $\ell_1 g \ell_2$, where $g \in M$, has size $|\ell_1 g \ell_2| = |g^{-1} \ell_1 g \ell_2| = |\ell_i \ell_2| = (s+1)^2$. 
Here the final equality holds because $|\ell_i \cap \ell_2| = 1$ (because distinct concurrent lines intersect in a unique point, in this case the point $1$). 
Since the double cosets of $\ell_1$ and $\ell_2$ partition $M$, it follows that $(s+1)^2$ divides $|M|=|\mathcal{P}|=(s+1)(st+1)$. 
Therefore, $s+1$ divides $st+1=(s+1)t - (t-1)$, and hence $s+1$ divides $t-1$, as claimed. 

Now suppose, towards a contradiction, that $\Inn(M)$ is transitive on $\{ \ell_1,\ldots,\ell_{t+1} \}$. 
Consider two lines incident with $1$, say $\ell_1,\ell_2$. 
Then a double coset $\ell_1 g \ell_2$, where $g \in M$, has size $(s+1)^2$ or $s+1$ according as $g^{-1}\ell_1 g \neq \ell_2$ or $g^{-1}\ell_1 g = \ell_2$. 
There are exactly $|M|/(t+1)$ elements $g\in M$ for which $g^{-1}\ell_1 g = \ell_2$, and since $\ell_1 h \ell_2 = \ell_1 g \ell_2$ if and only if $h \in \ell_1 g \ell_2$ (where $h \in M$), it follows that there are precisely $|M|/((s+1)(t+1))$ double cosets of size $s+1$. 
Therefore, $(s+1)(t+1)$ must divide $|M|=|\mathcal{P}|=(s+1)(st+1)$, and so $t+1$ must divide $st+1=(t+1)s - (s-1)$ and hence $s-1$. 
In particular, we have $s \ge t+2 > t$, and so \cite[2.2.2(i)]{MR2508121} implies that $\mathcal{S}$ cannot contain a subquadrangle of order $(s,1)$.
For a contradiction, we now construct such a subquadrangle. 

Consider the subset $\mathcal{P}'=\ell_1\ell_2$ of $\mathcal{P}=M$, let $\mathcal{L}' = \{ g_1\ell_2 \mid g_1 \in \ell_1\} \cup \{ \ell_1g_2 \mid g_2 \in \ell_2 \}$, and let $\inc'$ be the restriction of $\inc$ to $(\mathcal{P}' \times \mathcal{L}') \cup (\mathcal{L}' \times \mathcal{P}')$. 
We claim that $\mathcal{S}' = (\mathcal{P}',\mathcal{L}',\inc')$ is a subquadrangle of $\mathcal{S}$ of order $(s,1)$. 
First observe that, for each $\ell \in \mathcal{L}'$ and each $P \in \mathcal{P}'$ not incident with $\ell$, the unique point incident with $\ell$ and collinear with $P$ lies in $\mathcal{P}'$, because $\ell \subset P'$. 
Hence, $\mathcal{S}'$ satisfies the generalised quadrangle axiom. 
Now, every line in $\mathcal{L}'$ is incident with $s+1$ points in $\mathcal{P}'$, being a coset of either $\ell_1$ or $\ell_2$, so it remains to show that every point in $\mathcal{P}'$ is incident with exactly two lines in $\mathcal{L}'$. 
Given $P=g_1g_2 \in \mathcal{P}'$, where $g_1 \in \ell_1$, $g_2 \in \ell_2$, each line $\ell \in \mathcal{L}'$ incident with $P$ is either of the form $h_1 \ell_2$ for some $h_1 \in \ell_1$ or $\ell_1 h_2$ for some $h_2 \in \ell_2$, and since $P \in \ell$, we must have $h_1=g_1$ or $h_2=g_2$, respectively. 
Therefore, $P$ is incident with exactly two lines in $\mathcal{L}'$, namely $g_1 \ell_2$ and $\ell_1 g_2$.
\end{Prf}

We also check that, in the case of a thick generalised quadrangle, the assumption that $M$ is not semiregular on $\mathcal{L}$ is satisifed when $|M|$ is even. 

\begin{La} \label{GQsemireg}
Let $\mathcal{Q} = (\mathcal{P},\mathcal{L},\inc)$ be a thick generalised quadrangle of order $(s,t)$. 
Let $G$ be a group of collineations of $\mathcal{Q}$ that acts transitively on $\mathcal{L}$, and suppose that $G$ has a normal subgroup $M$ that acts regularly on $\mathcal{P}$. 
If $M$ has even order, then $M$ does not act semiregularly on $\mathcal{L}$.
\end{La}

\begin{Prf}
If $M_\ell$ is trivial for $\ell \in \mathcal{L}$, then $|\ell^M|=|M|=|\mathcal{P}|=(s+1)(st+1)$ divides $|\mathcal{L}|=(t+1)(st+1)$, and hence $s+1$ divides $t+1$, so \cite[Lemma~3.2]{Bamberg:2012yf} implies that $\gcd(s,t)>1$. 
However, $|M|$ is even, so $M$ contains an element of order $2$, and because $\gcd(s,t)>1$, it follows from \cite[Lemma~3.4]{Bamberg:2012yf} that every such element must fix some line, contradicting the assumption that $M_\ell$ is trivial.
\end{Prf}

\section{Proof of Theorem~\ref{mainGQthm}: HS type} \label{sec:proof}

Suppose that $\mathcal{Q}=(\mathcal{P},\mathcal{L},\inc)$ is a thick generalised quadrangle with a collineation group $G$ that acts transitively on $\mathcal{L}$ and primitively of O'Nan--Scott type HS on $\mathcal{P}$. 
Then
\[
T \rtimes \Inn(T) \le G \le T \rtimes \Aut(T)
\]
for some non-Abelian finite simple group $T$, with $T$ acting regularly on $\mathcal{P}$. 
Since $|T|$ is even by the Feit--Thompson Theorem~\cite{MR0166261}, Lemma~\ref{GQsemireg} tells us that $\mathcal{Q}$ satisfies the hypotheses of Theorem~\ref{prop:flagTrans} and Corollary~\ref{GQcor}. 
In particular, $s+1$ divides $t-1$ (by Corollary~\ref{GQcor}), and we write
\begin{equation} \label{eq:t'}
t' := \frac{t-1}{s+1}.
\end{equation}
Since $T$ acts regularly on $\mathcal{P}$, we have $|T|=|\mathcal{P}|=(s+1)(st+1)$.
By Higman's inequality, $t \le s^2$, and hence $t' \le s-1$. 
Therefore,
\[
|T| = (s+1)^2(st'+1) \quad \text{for some } 1 \le t' \le s-1.
\]
By Theorem~\ref{prop:flagTrans}, $G_1\le \Aut(T)$ acts transitively on the $t+1$ lines incident with $1$, and hence $t+1$ divides $|\Aut(T)| = |T|\cdot |\Out(T)|$. 
Therefore, $|\Out(T)|$ is divisible by $(t+1)/\gcd(t+1,|T|)$, so $t+1 \le \gcd(t+1,|T|)|\Out(T)|$. 
Since $|T| = (s+1)(st+1)$ is even, $s$ must be odd; and since $s+1$ divides $t-1$, we have $\gcd(t+1,s+1)=2$. 
Moreover, $st'+1=t-t'$, so $\gcd(t+1,st'+1) = \gcd(t+1,t-t') = \gcd(t+1,t'+1)$, and in particular $\gcd(t+1,|T|) \le 2^2(t'+1)$. 
Therefore, $t+1 \le 4(t'+1)|\Out(T)|$. 
Together with \eqref{eq:t'}, this implies $t'(s+1) + 2 \le 4(t'+1)|\Out(T)|$, and because $t'\ge 1$, it follows that
\[
s \le 8|\Out(T)| - 3.
\]
Since $|T| \le (s+1)(s^3+1)$ (by Higman's inequality), we have
\[
|T| \le (8|\Out(T)| - 2)((8|\Out(T)| - 3)^3+1).
\]
The following lemma therefore completes the proof of Theorem~\ref{mainGQthm} in the HS case.

\begin{La}\label{HS_hammer}
There is no finite non-Abelian simple group $T$ satisfying
\begin{enumerate}[(a)]
\item $|T| = (s+1)^2(st'+1)$, where $1 \le t' \le s-1$;\label{eq:|T|t'}
\item $2\le s \le 8|\Out(T)| - 3$; and \label{eq:sBound}
\item $|T| \le (8|\Out(T)| - 2)((8|\Out(T)| - 3)^3+1)$.\label{eq:|T|Bound}
\end{enumerate}
\end{La}

\begin{proof}
Since $(8x-2)((8x-3)^3+1) \leq (8x)^4$ for real $x\ge 1$, condition~\ref{eq:|T|Bound} implies that
\begin{equation} \label{eq:|T|Bound2}
|T| \le 2^{12} |\Out(T)|^4.
\end{equation}
We use \eqref{eq:|T|Bound2} instead of \ref{eq:|T|Bound} to rule out certain possibilities for $T$.
\begin{description}

\item[Case 1. $T\cong\Alt_n$ or a sporadic simple group]
If $T\cong\Alt_6$, then $|\Out(T)|=4$ and there is no solution to \ref{eq:|T|t'} subject to \ref{eq:sBound}. 
If $T$ is an alternating group other than $\Alt_6$, or a sporadic simple group, then $|\Out(T)|\le 2$, and so \ref{eq:|T|Bound} implies that $|T| \le (13+1)(13^3+1) = 30\;772$. 
This rules out everything except $T\cong\Alt_5$, $\Alt_7$ and $M_{11}$, and for these cases one checks that there is no solution to \ref{eq:|T|t'} subject to \ref{eq:sBound}. 

\item[Case 2. $T\cong A_1(q)$]
Suppose that $T\cong A_1(q)$, and write $q=p^f$ with $p$ prime and $f\ge 1$. 
Then $|T|=q(q^2-1)/\gcd(2,q-1)$, and $|\Out(T)|=\gcd(2,q-1)f$. 

Suppose first that $q$ is even, namely that $p=2$. 
Then $\gcd(2,q-1)=1$, and \ref{eq:|T|Bound} implies that
\[
2^f(2^{2f}-1) \le (8f-2)((8f-3)^3+1),
\]
which holds only if $f\le 7$. 
If $f=1$, then $T$ is not simple; and if $f=2$, then $T\cong \Alt_5$, which we have already ruled out. 
For $3 \le f \le 7$, there is no solution to \ref{eq:|T|t'} subject to \ref{eq:sBound}. 

Now suppose that $q=p^f$ is odd. 
Then $\gcd(2,q-1)=2$, and hence $|\Out(T)|=2f$, so \ref{eq:|T|Bound} reads
\[
p^f(p^{2f}-1) \le 2(16f-2)((16f-3)^3+1).
\]
If $f\ge 6$, then this inequality fails for all $p \ge 3$. 
The inequality holds if and only if
\[
q=p^f \in \{3,5,7,3^2,11,13,17,19,23,5^2,3^3,29,31,37,7^2,3^4,5^3,3^5\}.
\]
If $q=3$, then $T$ is not simple; if $q=5$, then $T\cong \Alt_5$, which we have ruled out; if $q=7$, then $T\cong A_2(2)$, which is ruled out in Case~3 below; and if $q=9$, then $T\cong \Alt_6$, which we have ruled out. 
For the remaining values of $q$, 
there is no solution to \ref{eq:|T|t'} subject to \ref{eq:sBound}. 

\item[Case 3. $T\cong A_n(q)$, $n\ge 2$]
Suppose that $T\cong A_n(q)$, with $n\ge 2$ and $q=p^f$. 
Then
\[
|T| = \frac{q^{n(n+1)/2}}{\gcd(n+1,q-1)} \prod_{i=1}^n (q^{i+1}-1), 
\]
and $|\Out(T)| = 2 \gcd(n+1,q-1) f$. 

First suppose that $n\ge 3$. 
Noting that $f = \log_p(q) = \ln(q)/\ln(p) \le \ln(q)/\ln(2)$ and $\gcd(n+1,q-1) \le q-1$, and applying \eqref{eq:|T|Bound2}, we find
\[
q^{n(n+1)/2} \prod_{i=1}^n (q^{i+1}-1) \le \frac{2^{16}}{\ln^4(2)}(q-1)^5\ln^4(q).
\]
This inequality fails for all $q\ge 2$ if $n=4$, and therefore fails for all $q\ge 2$ for every $n\ge 4$ (because the left-hand side is increasing in $n$ while the right-hand side does not depend on $n$). 
It fails for $n=3$ unless $q\in\{2,3\}$, but $A_3(2) \cong \Alt_8$ has already been ruled out, and \ref{eq:|T|Bound} rules out $A_3(3)$ because $|A_3(3)|=6\;065\;280 > 30(29^3+1) = 731\;700$. 

Finally, suppose that $n=2$. 
Noting that $\gcd(3,q-1) \le 3$ and $f \le \ln(q)/\ln(2)$, \eqref{eq:|T|Bound2} gives
\[
q^3(q^2-1)(q^3-1) \le \frac{2^{16}3^5}{\ln^4(2)} \ln^4(q).
\]
This implies that $q\le 15$.
For $q\in\{5,8,9,11,13\}$, the sharper inequality \ref{eq:|T|Bound} fails.
For $q\in\{2,3,4,7\}$, there are no solutions to \ref{eq:|T|t'} subject to \ref{eq:sBound}. 

\item[Case 4. $T\cong {}^2A_n(q^2)$]
Suppose that $T\cong {}^2A_n(q^2)$, where now $q^2=p^f$ for some prime $p$ and $f\geq 1$. 
We have $n\ge 2$, 
\[
|T| = \frac{q^{n(n+1)/2}}{\gcd(n+1,q+1)} \prod_{i=1}^n (q^{i+1}-(-1)^{i+1}),
\]
and $|\Out(T)| = \gcd(n+1,q+1)f$.  

First suppose that $n \ge 4$. 
Noting that $f = \log_p(q) = \ln(q^2)/\ln(p) \le 2\ln(q)/\ln(2)$, and that $\gcd(n+1,q+1)\le q+1$, \eqref{eq:|T|Bound2} gives
\[
q^{n(n+1)/2} \prod_{i=1}^n (q^{i+1}-(-1)^{i+1}) \le \frac{2^{16}}{\ln^4(2)}(q+1)^5\ln^4(q).
\]
This inequality fails for all $q\ge 2$ for $n=4$, and hence fails for all $q\ge 2$ for every $n\ge 4$. 

Now suppose that $n=3$. 
Then we can replace the $(q+1)^5$ on the right-hand side above by $4^5=2^{10}$, because $\gcd(n+1,q+1)=\gcd(4,q+1)\le 4$. 
This yields
\[
q^6(q^2-1)(q^3+1)(q^4-1) \le \frac{2^{26}}{\ln^4(2)}\ln^4(q),
\]
which implies that $q\le 4$. 
If $q\in\{2,3\}$, then there are no solutions to \ref{eq:|T|t'} subject to \ref{eq:sBound}. 
If $q=4$, then \ref{eq:|T|Bound} fails.

Finally, suppose that $n=2$. 
Then $\gcd(n+1,q+1) \le 3$, and hence 
\[
q^3(q^2-1)(q^3+1) \le \frac{2^{16}3^5}{\ln^4(2)} \ln^4(q),
\]
which implies that $q \le 15$.
If $q=2$, then $T\cong {}^2A_2(2^2)$ is not simple. 
If $q\in\{3,4,5,8\}$, then there are no solutions to \ref{eq:|T|t'} subject to \ref{eq:sBound}. 
If $q\in\{7,9,11,13\}$, then \ref{eq:|T|Bound} fails. 

\item[Case 5. Remaining possibilities for $T$] \label{HSallotherT}
We now rule out the remaining possibilities for the finite simple group $T$.

(i) {\em $T\cong B_n(q)$ or $C_n(q)$.} 
First suppose that $T\cong C_n(q)$, and write $q=p^f$ with $p$ prime and $f\ge 1$. 
We have $n\ge 3$, $|T|=q^{n^2}/\gcd(2,q-1)\cdot \prod_{i=1}^n(q^{2i}-1)$, and $|\Out(T)| = \gcd(2,q-1)f$. 
Noting that $f \le \ln(q)/\ln(2)$ and $\gcd(2,q-1) \le 2$, \eqref{eq:|T|Bound2} implies that
\[
q^{n^2} \prod_{i=1}^n(q^{2i}-1) \le \frac{2^{17}}{\ln^4(2)} \ln^4(q).
\]
However, this inequality fails for all $q\ge 2$ if $n=3$, and hence fails for all $q\ge 2$ for every $n \ge 3$.

Now suppose that $T\cong B_n(q)$, writing $q=p^f$ as before. 
In this case we have $n\ge 2$, and again $|T|=q^{n^2}/\gcd(2,q-1)\cdot \prod_{i=1}^n(q^{2i}-1)$. 
If $n\ge 3$ and $q$ is even, then $B_n(q) \cong C_n(q)$. 
If $n\ge 3$ and $q$ is odd, then $|\Out(T)|$ is the same as for $C_n(q)$. 
We may therefore assume that $n=2$. 
First suppose that $q=2^f$. 
Then $|\Out(T)| = 2\gcd(2,q-1)f = 2f$, so \eqref{eq:|T|Bound2} implies that
\begin{equation} \label{B2lastcase}
2^{4f}(2^{2f}-1)(2^{4f}-1) \le 2^{16}f^4,
\end{equation}
and hence $f\in\{1,2\}$. 
For $f=1$, $B_2(2)$ is not simple but its derived subgroup $B_2(2)'\cong \Alt_6$ is simple and has already been ruled out. 
For $f=2$, \ref{eq:|T|Bound} fails. 
Now suppose that $q$ is odd. 
Then $|\Out(T)| = \gcd(2,q-1)f = 2f$ and $f \le \ln(q)/\ln(3)$, so \eqref{eq:|T|Bound2} implies that
\[
q^4(q^2-1)(q^4-1) \le \frac{2^{17}}{\ln^4(3)} \ln^4(q),
\]
and hence $q=3$. 
However, $B_2(3) \cong {}^2A_3(2^2)$, which has been dealt with in Case~4.

(ii) {\em $T\cong D_n(q)$.}
Suppose that $T\cong D_n(q)$, writing $q=p^f$ again. 
We have $n\ge 4$, $|T| = q^{n(n-1)} (q^n-1) / \gcd(4,q^n-1) \cdot \prod_{i=1}^{n-1} (q^{2i}-1)$, and 
\[
|\Out(T)| = \begin{cases} 
6 \gcd(2,q-1)^2 f & \text{if } n=4 \\
2 \gcd(2,q-1)^2 f & \text{if } n<4 \text{ and } n \text{ is even} \\
2 \gcd(4,q^n-1) f & \text{if } n<4 \text{ and } n \text{ is odd}.
\end{cases}
\] 
If $q$ is odd, then $\gcd(4,q^n-1) \le 4$, $|\Out(T)| \le 24f$, and $f\le \ln(q)/\ln(3)$, so \eqref{eq:|T|Bound2} implies that
\[
q^{n(n-1)} (q^n-1) \prod_{i=1}^{n-1} (q^{2i}-1) \le \frac{2^{26}3^4}{\ln^4(3)} \ln^4(q),
\]
which fails for all $q\ge 3$ if $n=4$, and hence fails for all $q\ge 3$ for every $n\ge 4$. 
If $q$ is even, then $\gcd(4,q^n-1) =1$, $|\Out(T)| \le 6f$ and $f=\ln(q)/\ln(2)$, so \eqref{eq:|T|Bound2} implies that
\[
q^{n(n-1)} (q^n-1) \prod_{i=1}^{n-1} (q^{2i}-1) \le \frac{2^{16}3^4}{\ln^4(2)} \ln^4(q),
\]
which fails for all $q\ge 2$ if $n=4$, and hence fails for all $q\ge 2$ for every $n\ge 4$. 

(iii) {\em $T\cong E_6(q), E_7(q), E_8(q)$ or $F_4(q)$.}
Suppose that $T$ is one of $E_6(q)$, $E_7(q)$, $E_8(q)$ or $F_4(q)$, and write $q=p^f$ again. 
Observe that $|E_i(q)| \ge |F_4(q)|$ for every $i\in\{6,7,8\}$, for all $q\ge 2$. 
Hence
\[
|T| \ge |F_4(q)| = q^{24}(q^{12}-1)(q^8-1)(q^6-1)(q^2-1) \ge \frac{q^{52}}{2^4}.
\]
Since $|\Out(T)| \le 2\gcd(3,q-1) f \le 6\ln(q)/\ln(2)$, \eqref{eq:|T|Bound2} implies the following inequality, which fails for all $q\ge 2$:
\[
q^{52} \le \frac{2^{20}3^4}{\ln^4(2)} \ln^4(q).
\]

(iv) {\em $T\cong G_2(q)$.}
Suppose that $T\cong G_2(q)$, with $q=p^f$. 
Then $|T| = q^6(q^6-1)(q^2-1)$. 
If $p=3$, then $|\Out(T)| = 2f$, so \ref{eq:|T|Bound} implies that $3^{6f}(3^{6f}-1)(3^{2f}-1) \le 2^{16}f^4$, which fails for all $f\ge 1$. 
If $p \neq 3$, then $|\Out(T)| = f \le \ln(q)/\ln(2)$, and \eqref{eq:|T|Bound2} implies the following inequality, which fails for all $q\ge 2$:
\[
q^6(q^6-1)(q^2-1) \le \frac{2^{12}}{\ln^4(2)} \ln^4(q).
\]
Note that $G_2(2)$ is not simple, but $G_2(2)' \cong {}^2A_2(3^2)$ is simple and has already been ruled out.

(v) {\em $T\cong {}^2D_n(q)$.}
Suppose that $T\cong {}^2D_n(q^2)$, now writing $q^2=p^f$. 
Then $n\ge 4$, 
\[
|T| = \frac{q^{n(n-1)} (q^n+1)}{\gcd(4,q^n+1)} \prod_{i=1}^{n-1} (q^{2i}-1),
\] 
and $|\Out(T)| = \gcd(4,q^n+1)f$. 
Since $f \le 2\ln(q)/\ln(2)$ and $\gcd(4,q^n+1) \le 4$, \eqref{eq:|T|Bound2} implies that
\[
q^{n(n-1)} (q^n+1) \prod_{i=1}^{n-1} (q^{2i}-1) \le \frac{2^{26}}{\ln^4(2)} \ln^4(q).
\]
This fails for all $q\ge 2$ if $n=4$, and hence fails for all $q\ge 2$ for every $n\ge 4$.

(vi) {\em $T\cong {}^2E_6(q^2)$.} 
Suppose that $T\cong {}^2E_6(q^2)$, with $q^2=p^f$. 
Then
\[
|T|=\frac{1}{(3,q+1)} q^{36}(q^{12}-1)(q^9+1)(q^8-1)(q^6-1)(q^5+1)(q^2-1),
\]
and $|\Out(T)| = \gcd(3,q+1)f$. 
Noting that $f \le 2\ln(q)/\ln(2)$ and $\gcd(3,q+1) \le 3$, \eqref{eq:|T|Bound2} implies the following inequality, which fails for all $q\ge 2$:
\[
q^{36}(q^{12}-1)(q^9+1)(q^8-1)(q^6-1)(q^5+1)(q^2-1) \le \frac{3^5 2^{16}}{\ln^4(2)} \ln^4(q).
\]

(vii) {\em $T\cong {}^3D_4(q^3)$.} 
Suppose that $T\cong {}^3D_4(q^2)$, where now $q^3=p^f$. 
Then 
\[
|T|=q^{12}(q^8+q^4+1)(q^6-1)(q^2-1), 
\]
and $|\Out(T)| = f = 3\ln(q)/\ln(p) \le 3\ln(q)/\ln(2)$, so \eqref{eq:|T|Bound2} implies the following inequality, which fails for all $q\ge 2$:
\[
q^{12}(q^8+q^4+1)(q^6-1)(q^2-1) \le \frac{3^4 2^{12}}{\ln^4(2)} \ln^4(q).
\]

(viii) {\em $T\cong {}^2B_2(q)$, ${}^2G_2(q)$, or ${}^2F_4(q)$.}
Finally, suppose that $T$ is as in one of the lines of Table~\ref{tab:Suz}. 
Suppose first that $n\ge 1$. 
Then $|\Out(T)|=2n+1$ in each case, and \eqref{eq:|T|Bound2} therefore implies that $|T| \le 2^{12}(2n+1)^4$. 
This inequality holds only in the case $T\cong {}^2B_2(2^{2n+1})$ with $n=1$, but $|^2B_2(2^3)| = 29\;120$ cannot be written in the form \ref{eq:|T|t'} subject to \ref{eq:sBound}. 
For $n=0$, we have that ${}^2B_2(q)$ is not simple; ${}^2G_2(3)$ is not simple, but ${}^2G_2(3)' \cong A_1(8)$ has been ruled out in Case~2 above; and $^2F_4(2)$ is not simple, but $^2F_4(2)'$ is simple of order $17\;971\;200$ and has outer automorphism group of order $2$, so \eqref{eq:|T|Bound2} fails.
\end{description}

This completes the proof of Lemma~\ref{HS_hammer}.
\end{proof}

\begin{table}[!t]
\begin{center}
\begin{tabular}{lll}
\toprule
$T$ & $|T|$ & $q$ \\
\midrule
$^2B_2(q)$ & $q^2(q^2+1)(q-1)$ & $2^{2n+1}$ \\ 
$^2G_2(q)$ & $q^3(q^3+1)(q-1)$ & $3^{2n+1}$ \\ 
$^2F_4(q)$ & $q^{12}(q^6+1)(q^4-1)(q^3+1)(q-1)$ & $2^{2n+1}$ \\ 
\bottomrule
\end{tabular}
\end{center}
\caption{Orders of the Suzuki and Ree simple groups.}\label{tab:Suz}
\end{table}

\section{Proof of Theorem~\ref{mainGQthm}: HC type} \label{sec:proofHC}

Suppose that $\mathcal{Q}=(\mathcal{P},\mathcal{L},\inc)$ is a thick generalised quadrangle with a collineation group $G$ that acts transitively on $\mathcal{L}$ and primitively of O'Nan--Scott type HC on $\mathcal{P}$. 
Then
\[
M \rtimes \operatorname{Inn}(M) \le G \le M \rtimes \operatorname{Aut}(M),
\]
where $M=T_1\times \cdots \times T_k$, with $k\ge 2$ and $T_1 \cong \cdots \cong T_k \cong T$ for some non-Abelian finite simple group $T$. 
Moreover, $M$ acts regularly on $\mathcal{P}$, and $G$ induces a subgroup of $\text{Aut}(T) \wr S_k$ which acts transitively on the set $\{T_1,\ldots,T_k\}$ (see \cite[Section 5]{MR1477745}).
Since $|M|=|T|^k$ is even by the Feit--Thompson Theorem~\cite{MR0166261}, Lemma~\ref{GQsemireg} tells us that $\mathcal{Q}$ satisfies the hypotheses of Theorem~\ref{prop:flagTrans} and Corollary~\ref{GQcor}. 
In particular, $s+1$ divides $t-1$ (by Corollary~\ref{GQcor}), and we define $t'$ as in \eqref{eq:t'}.

We first rule out the case $k\ge 3$, and then deal with the case $k=2$ separately.

\subsection{$k\ge 3$}

Suppose, towards a contradiction, that $k\ge 3$. 
Denote by $\ell_1,\ldots,\ell_{t+1}$ the lines incident with the identity $1 \in M$. 
By Lemma~\ref{lineUnionCosets}(ii), we may identify $\ell_i$ with the subgroup of $M$ comprising all points incident with $\ell_i$.
Let us write $\ell := \ell_1$ for brevity. 

\begin{Claim} \label{claim1}
$M$ cannot be decomposed in the form $M=A\times B$ with $\ell \cap A \neq \{1\}$ and $\ell \cap B \neq \{1\}$. 
\end{Claim}

\begin{Prf}
Suppose, towards a contradiction, that $M=A\times B$ with $\ell \cap A \neq \{1\}$ and $\ell \cap B \neq \{1\}$. 
We may assume, without loss of generality, that (i) $A$ contains $T_1$, and (ii) $\ell \cap A$ contains an element $x=(x_1,\ldots,x_k)$ that projects non-trivially onto each simple direct factor of $A$ (if not, then change the decomposition of $M$ to $A'\times B'$ with $A'\le A$ and $B'\ge B$). 
Take also $y \in \ell \cap B$ with $y \neq 1$. 
For every $a \in \text{Inn}(A) \le \text{Inn}(M)$, we have $y^a=y$ and hence $\ell^a=\ell$, because $a$ also fixes the point $1\in \ell$. 
In particular, $\ell$ is fixed by every element of $\Inn(T_1)$, regarded as a subgroup of $\Inn(A)$. 
Therefore, $(z,x_2,\ldots,x_k) \in \ell$ for all $z \in x_1^{T_1}$, and hence $\ell$ contains the group $\ell_0 := \langle (z,x_2,\ldots,x_k) : z \in x_1^{T_1} \rangle$. 
Let $\pi_1$ denote the projection onto $T_1$. 
Then $\pi_1(\ell_0) = \langle z : z \in x_1^{T_1} \rangle = T_1$, and hence $\pi_1(\ell) = T_1$. 
Also, taking $z \neq x_1$, we see that $\ell \cap T_1$ contains $(z,x_2,\ldots,x_k)^{-1}x = (z^{-1}x_1,1,\ldots,1) \neq 1$. 
That is, $\ell \cap T_1$ is non-trivial, and it is normal in the simple group $\pi_1(\ell)=T_1$, so $\ell \cap T_1 = T_1$ and hence $T_1 \le \ell$. 
Now, $G_1$ acts transitively on both $\{T_1,\ldots,T_k\}$ (because $G$ is transitive on $\{T_1,\ldots,T_k\}$ and $G=MG_1$) and $\{\ell_1,\ldots,\ell_{t+1}\}$ (because $G$ is flag-transitive, by Theorem~\ref{prop:flagTrans}). 
Therefore, $t+1$ divides $k$, and, without loss of generality, $\ell=\ell_1$ contains $T_{U_1} := T_1\times \cdots \times T_{k/(t+1)}$, $\ell_2$ contains $T_{U_2} := T_{k/(t+1)+1}\times \cdots \times T_{2k/(t+1)}$, and so on.

{\em Sub-claim.} $\ell = T_{U_1}$. 

{\em Proof of sub-claim.} 
It remains to show that $T_{U_1}$ contains $\ell$. 
Suppose, towards a contradiction, that there exists $w \in \ell \setminus T_{U_1}$. 
Then there exists $i > k/(t+1)$ such that the $i$th component $w_i$ of $w$ is non-trivial, and so there exists $\sigma \in \text{Inn}(T_i)$ such that $w_i^\sigma \neq w_i$. 
Regarding $\sigma$ as an element of $\Inn(M) \le G_1$, we see that $\sigma$ fixes $\ell$, because it centralises $T_1 \le \ell$. 
Hence, $w^\sigma \in \ell$, and so $\ell$ contains $w^{-1}w^\sigma \in \ell\cap T_i \setminus \{1\}$. 
However, $T_i \le T_{U_j} \le \ell_j$ for some $j\neq 1$, and hence $\ell$ intersects $\ell_j$ in more than one point, a contradiction, proving the sub-claim.

By the sub-claim, $s+1=|T|^u$, where $u=k/(t+1)$. 
Since $|T|^{(t+1)u} = |M|$, we have $(s+1)^{t+1} = (s+1)^2(st'+1)$, where $t' := (t-1)/(s+1) \le s-1$ as before. 
Since $st'+1 \le s(s-1)+1 < (s+1)^2$, this implies that $(s+1)^{t-1} < (s+1)^2$, so $t=2$, and hence $s+1 \mid t-1=1$, a contradiction.
\end{Prf}

\begin{Claim} \label{claim2}
$\ell$ is isomorphic to a subgroup of $T$. 
\end{Claim}

\begin{Prf}
Let $x \in \ell \setminus \{1\}$ have minimal support $U$.
Suppose, without loss of generality, that $x_1 := \pi_1(x) \neq 1$. 
Suppose further, towards a contradiction, that there exists $y \in \ell \setminus \{1\}$ with $\pi_1(y)=1$. 
Then every $a \in \Inn(T_1)$ fixes $y$ and hence fixes $\ell$, so $\ell$ contains $x^a$ and therefore contains $x^a x^{-1} \in T_1\cap \ell$. 
Taking $a$ not in $C_T(x_1)$ makes $x^a x^{-1}$ non-trivial, and the minimality of the support $U$ of $x$ implies that $U=\{1\}$, so $x \in T_1$. 
However, the existence of $y$ now contradicts Claim~\ref{claim1}, because taking $A=T_1$ and $B=T_2\times\cdots\times T_k$ gives $x \in \ell\cap A$ and $y \in \ell\cap B$. 
Hence, if $x$ has minimal support $U$ containing $1$, then every non-trivial element of $\ell$ must project non-trivially onto $T_1$. 
Therefore, $\ell$ is isomorphic (under projection) to a subgroup of $T_1$. 
\end{Prf}
 
We now use Claim~\ref{claim2} to derive a contradiction to the assumption that $k\ge 3$. 
By Claim~\ref{claim2}, $s+1=|\ell|$ divides $|T|$, so in particular $s+1\le |T|$. 
Writing $|M|=(s+1)^2(st'+1)$ with $t':=(t-1)/(s+1)\le s-1$ as before, we have $(s+1)^2 > s(s-1)+1 \ge st'+1 = |M|/(s+1)^2 \ge |M|/|T|^2 = |T|^{k-2} \ge (s+1)^{k-2}$, and hence $2>k-2$, namely $k\le 3$.

Now suppose, towards a contradiction, that $k=3$. 
Write $|T|=n(s+1)$. 
Then $st+1 = |M|/(s+1) = |T|^3/(s+1) = n^3(s+1)^2$, and hence $n^3 \equiv 1 \pmod s$. 
On the other hand, $s^3+1 \ge st+1 = n^3 (s+1)^2 > n^3s^2+1$, so $n^3<s$. 
Therefore, $n=1$, so $|T|=s+1$ and $t=s+2$. 
Together with Claim~\ref{claim2}, this implies that $\ell$ is isomorphic to $T$. 
Consider first the case where $\ell$ is a diagonal subgroup $\{(t,t^a, t^b) : t\in T\} \le M$ for some $a,b\in \text{Aut}(T)$.
As $(c,d)\in \Inn(T_2) \times \Inn(T_3)\le G_1$ runs over all possibilities, we obtain $|T|^2$ distinct images $\ell^{(c,d)} = \{ (t,t^{ac},t^{bd}) : t\in T\}$ of $\ell$. 
Indeed, if $\ell=\ell^{(c,d)}$, then $t^a=t^{ac}$ for all $t \in T$, or equivalently, $u=u^c$ for all $u \in T$; that is, $c$ is the identity automorphism of $T$ (and similarly, $d$ is the identity). 
Hence, $s+3=t+1\geq (s+1)^2$, a contradiction.
Now consider the case where $\ell$ is a diagonal subgroup $\{(t,t^a, 1) : t\in T\} \le T_1 \times T_2$ for some $a\in \text{Aut}(T)$.
Then $3$ divides $t+1$ because $G_1$ is transtive on the $T_i$, and we have exactly $(t+1)/3$ lines incident with $1$ that are diagonal subgroups of $T_1 \times T_2$. 
As $c\in \text{Inn}(T_2) \le G_1$ runs over all possibilities, we obtain $|T|$ distinct images $\ell^c = \{ (t,t^{ac},1) : t\in T\}$ of $\ell$. 
Hence, $(s+3)/3=(t+1)/3\ge s+1$, a contradiction.
This leaves only the possibility that $\ell\le T_1$, and hence $\ell=T_1$ because $|\ell|=s+1=|T_1|$. 
This implies that $t+1=3$, and hence $s=0$ because $s+1$ divides $t-1$, a contradiction.

\subsection{$k=2$}

Here we argue as in the case where the primitive action of $G$ on $\mathcal{P}$ has type HS. 
That is, we obtain an upper bound on $|T|$ in terms of $|\Out(T)|$, and consider the possibilities for $T$ case by case using the Classification of Finite Simple Groups. 
We have $M=T_1 \times T_2 \cong T^2$, and
\[
|M| = (s+1)(st+1) = (s+1)^2(st'+1), \quad \text{where } 1\le t'\le s-1.
\]
Therefore,
\[
|T| = (s+1)(st'+1)^{1/2}, \quad \text{where } 1\le t'\le s-1 \text{ and } st'+1 \text{ is a square.}
\]
Writing $y^2=st'+1$, this is equivalent to
\[
|T| = (s+1)y, \quad \text{where } 3\le y^2\le s(s-1)+1\text{ and } s\mid y^2-1.
\]
By Theorem~\ref{prop:flagTrans}, $G_1\le \text{Aut}(M) \cong \text{Aut}(T) \wr S_2$ acts transitively on the lines incident with $1$, and hence $t+1$ divides $|\text{Aut}(M)| = 2 |T|^2 |\Out(T)|^2$. 
Therefore, $|\Out(T)|^2$ is divisible by 
\[
\frac{t+1}{\gcd(t+1,2|T|^2)} = \frac{t+1}{\gcd(t+1,2(s+1)^2(st'+1))}.
\] 
In particular, $t+1 \le \gcd(t+1,2|T|^2)|\Out(T)|^2$. 
We have (i) $\gcd(t+1,s+1)=2$, so $\gcd(t+1,2(s+1)^2) \le 8$; and (ii) $\gcd(t+1,st'+1) = \gcd(t+1,t'+1)$. 
Hence, $\gcd(t+1,2|T|^2) \le 8(t'+1)$, and so $t+1 \le 8(t'+1)|\Out(T)|^2$. 
Re-writing this as $t'(s+1) + 2 \le 8(t'+1)|\Out(T)|^2$, and noting that $t'\ge 1$, we obtain
\[
s \le 16|\Out(T)|^2 - 3.
\]
Higman's inequality then gives
\[
|T|^2 = |M| \le (16|\Out(T)|^2 - 2)((16|\Out(T)|^2 - 3)^3+1).
\]
The following lemma therefore rules out all but two possibilities for $T$.

\begin{La}\label{HC_hammer}
Let $T$ be a finite non-Abelian simple group satisfying
\begin{enumerate}[(a)]
\item $|T| = (s+1)y$, where $3\le y^2\le s(s-1)+1$ and $s\mid y^2-1 $; \label{eq:|T|t'2}
\item $2\le s \le 16|\Out(T)|^2 - 3$; and \label{eq:sBoundk=2}
\item $|T|^2 \le (16|\Out(T)|^2 - 2)((16|\Out(T)|^2 - 3)^3+1)$. \label{eq:|T|Boundk=2}
\end{enumerate}
Then either \textnormal{(i)} $T\cong \Alt_6$, $s=19$, and $y=18$; or \textnormal{(ii)} $T\cong A_2(2)$, $s=13$, and $y=12$.
\end{La}

\begin{proof}
The right-hand side of (c) is at most $(16|\Out(T)|^2)^4$, so
\begin{equation} \label{eq:|T|Boundk=2:2}
|T| \le 2^8 |\Out(T)|^4.
\end{equation}
Since \eqref{eq:|T|Boundk=2:2} implies \eqref{eq:|T|Bound2}, any group $T$ that was ruled out using \eqref{eq:|T|Bound2} in the HS case (that is, in the proof of Lemma~\ref{HS_hammer}) is automatically ruled out here.
To rule out the remaining possibilities for $T$, we use either \eqref{eq:|T|Boundk=2:2} or \ref{eq:|T|Boundk=2}, or check that \ref{eq:|T|t'2} has no solution subject to \ref{eq:sBoundk=2}. Note that \ref{eq:|T|t'2} implies $y\le s< y^2$.

\begin{description}
\item[Case 1. $T\cong \Alt_n$ or a sporadic simple group]
If $T$ is an alternating group other than $\Alt_6$, or a sporadic simple group, then $|\Out(T)|\le 2$ and so \ref{eq:|T|Boundk=2} implies that $|T| < 3\;752$. 
Hence, $T$ is one of $\Alt_5$, $\Alt_6$, or $\Alt_7$.
If $T\cong \Alt_5$, then by \ref{eq:|T|t'2}, we have $(s+1)y=60$ and $s\mid y^2-1$,
which is impossible.
If $T\cong \Alt_7$, then we again apply \ref{eq:|T|t'2}: $(s+1)y=2520$, $s\mid y^2-1$, and $y^2\le s(s-1)+1$, which is again impossible.
Finally, we examine the case $T\cong \Alt_6$, where $|\Out(T)|=4$.
Applying \ref{eq:|T|t'2}, we have $s=19$, $y=18$ as the only valid solution. 

\item[Case 2. $T\cong A_1(q)$]
Suppose that $T\cong A_1(q)$, and write $q=p^f$ with $p$ prime and $f\ge 1$. 
Then $|T|=q(q^2-1)/(2,q-1)$, and $|\Out(T)|=(2,q-1)f$. 

Suppose first that $q$ is even, namely that $p=2$. 
Then $\gcd(2,q-1)=1$, and \ref{eq:|T|Bound} implies that
\[
2^{2f}(2^{2f}-1)^2 \le   (16 f^2-2) ((16 f^2-3)^3+1),
\]
which holds only if $f\le 7$. 
If $f=1$, then $T$ is not simple; and if $f=2$, then $T\cong \Alt_5$, which we have already ruled out. 
For $3 \le f \le 7$, there is no solution to \ref{eq:|T|t'2} subject to \ref{eq:sBoundk=2}. 

Now suppose that $q=p^f$ is odd. 
Then $\gcd(2,q-1)=2$, and hence $|\Out(T)|=2f$. 
By \ref{eq:|T|Bound}, we have
\[
p^{2 f} (p^{2 f}-1)^2\le 8 (32 f^2-1) ((64 f^2-3)^3+1),
\]
which implies that either $11\le p\le 19$ and $f=1$; $5\le p\le 7$ and $f\le 2$; or $p=3$ and $f\le 4$.
If $q=3$, then $T$ is not simple; if $q=5$, then $T\cong \Alt_5$, which we have ruled out; if $q=7$, then $T\cong A_2(2)$, which is ruled out in Case~3 below; and if $q=9$, then $T\cong \Alt_6$, which we have already dealt with in Case~1. Hence, we only need to consider $q\in\{11,13,17,19, 3^3,3^4, 5^2,7^2\}$. 
For each of these values, there is no solution to \ref{eq:|T|t'2} subject to \ref{eq:sBoundk=2}.

\item[Case 3. $T\cong A_n(q)$, $n\ge 2$]
Since \eqref{eq:|T|Boundk=2:2} implies \eqref{eq:|T|Bound2}, by comparing with the proof of Case~2 in Lemma~\ref{HS_hammer}, we see that we only need to check $T\cong A_3(3)$, and $T\cong A_2(q)$ for $q \le 13$. 
The former is ruled out by \eqref{eq:|T|Boundk=2:2}, because $|A_3(3)|=6\;065\;280 > 2^84^4 = 65\;536$. 
For $T\cong A_2(q)$, 
\ref{eq:|T|Boundk=2} implies that
\[
q^6 (q^2-1)^2 (q^3-1)^2 \le 9 \left( \frac{576}{\ln^2(2)}\ln^2(q) - 2 \right) \left( \left( \frac{576}{\ln^2(2)}\ln^2(q) - 3 \right)^3 + 1 \right)
\]
Therefore, $q\le 10$. 
For $q=2$, there is a unique solution to \ref{eq:|T|t'2} subject to \ref{eq:sBoundk=2}, namely $s=13$, $t'=11$. 
For $q \in \{3,4,5,7,8,9\}$, there are no solutions to \ref{eq:|T|t'2} subject to \ref{eq:sBoundk=2}. 

\item[Case 4. $T\cong {}^2A_n(q^2)$]
Since \eqref{eq:|T|Boundk=2:2} implies \eqref{eq:|T|Bound2}, we only need to check $T\cong {}^2A_3(q^2)$ for $2 \le q \le 4$, and $T\cong {}^2A_2(q^2)$ for $q\le 13$. 
If $(n,q) = (3,3)$ or $(3,4)$, then \eqref{eq:|T|Boundk=2:2} fails; and for $(n,q)=(3,2)$, there are no solutions to \ref{eq:|T|t'2} subject to \ref{eq:sBoundk=2}. 
For $n=2$, \ref{eq:|T|Boundk=2} gives
\[
q^6 (q^2-1)^2 (q^3+1)^2 \le 9 \left( \frac{576}{\ln^2(2)}\ln^2(q) - 2 \right) \left( \left( \frac{576}{\ln^2(2)}\ln^2(q) - 3 \right)^3 + 1 \right),
\]
and hence $q\le 10$. 
If $q=2$, then $T\cong {}^2A_2(q^2)$ is not simple. 
If $q\in\{3,4,5,7,8,9\}$, then there are no solutions to \ref{eq:|T|t'2} subject to \ref{eq:sBoundk=2}. 

\item[Case 5. Remaining possibilities for $T$]
We only need to check the groups from Case~5 of the proof of Lemma~\ref{HS_hammer} that were not ruled out by \eqref{eq:|T|Bound2} or by exceptional isomorphisms to groups that have already been handled. 
There are only two such cases. 
If $T\cong B_2(2^f)$ with $f=2$, then, using \eqref{eq:|T|Boundk=2:2} instead of \eqref{eq:|T|Bound2}, the $2^{16}$ on the right-hand side of \eqref{B2lastcase} becomes $2^{12}$, and the resulting inequality $2^{4f}(2^{2f}-1)(2^{4f}-1) \le 2^{12}f^4$ fails when $f=2$. 
If $T\cong {}^2B_2(2^{2n+1})$ with $n=1$, then \eqref{eq:|T|Boundk=2:2} fails (although \eqref{eq:|T|Bound2} does not).
\end{description}

This completes the proof of Lemma~\ref{HC_hammer}.
\end{proof}

It remains to rule out cases (i) and (ii) from Lemma~\ref{HC_hammer}. 
Using $y^2=st'+1$, we find that $t=341$ in case (i), and $t=155$ in case (ii). 
Both cases are then ruled out because the required divisibility condition $t+1 \mid |\Aut(M)|=2|T|^2|\Out(T)|^2$ fails. 
(Note that $|\Aut(M)| = 4\;147\;200$ if $T\cong \Alt_6$, and $|\Aut(M)| = 225\;792$ if $T\cong A_2(2)$.)

\section*{Acknowledgements}

We thank the referee for a careful reading of the paper, and in particular for bringing to our attention an error in a previous version of the proof of Corollary~\ref{GQcor}.

\end{document}